\documentclass[11pt,a4paper]{amsart}
\usepackage{amsmath,amsthm,amssymb,latexsym}
\usepackage{graphicx}

\newcommand{\bd}{{\mathbb{D}}}

\newcommand{\bc}{{\mathbb{C}}}

\newcommand{\bt}{{\mathbb{T}}}

\newcommand{\css}{{\mathcal{S}}}

\newcommand{\s}{\sigma}

\renewcommand{\ss}{\Sigma}

\newcommand{\g}{\gamma}
\renewcommand{\gg}{\Gamma}
\newcommand{\eps}{\varepsilon}
\newcommand{\z}{\zeta}

\newcommand{\noe}{\not\equiv}

\newcommand{\ovl}{\overline}

\allowdisplaybreaks
\numberwithin{equation}{section}

\newtheorem{theorem}{Theorem}[section]

\newtheorem{proposition}[theorem]{Proposition}

\theoremstyle{definition}
\newtheorem{definition}[theorem]{Definition}
\newtheorem{remark}[theorem]{Remark}

\newtheorem{example}[theorem]{Example}

\begin{document}

\title[Darlington synthesis]
{On a local Darlington synthesis problem}
\author[L. Golinskii]{L. Golinskii}

\address{B. Verkin Institute for Low Temperature Physics and
Engineering, 47 Science ave., Kharkov 61103, Ukraine}
\email{golinskii@ilt.kharkov.ua}

\date{\today}


\maketitle

\begin{abstract}
The Darlington synthesis problem (in the scalar case) is the problem of embedding a given contractive analytic function to an inner $2\times 2$
matrix function as the entry. A fundamental result of Arov--Douglas--Helton relates this algebraic property to a pure analytic one known as
a pseudocontinuation of bounded type. We suggest a local version of the Darlington synthesis problem and prove a local analog of the ADH theorem.
\end{abstract}

\section*{Introduction}
\label{s0}

The Darlington synthesis with its origin in electrical engineering has a long history. The synthesis of non-lossless circuits was a hard problem at
the pre-computer time. The idea of the Darlington synthesis was to reduce any such problem to a lossless one

A mathematical setup in the simplest scalar case looks as follows, see \cite{ar71, ar73, duhel, ga05} and \cite[Section 6.7]{RS02}.

An analytic function $s$ on the unit disk $\bd$ is called a {\it Schur (contractive)} function, $s\in\css$, if $|s|\le1$ in $\bd$. Similarly, an analytic
on $\bd$ $2\times 2$ matrix function $S$ (throughout this note we deal only with matrices of order $2$) is a Schur (contractive) matrix function,
$S\in\css^{(m)}$, if
$$  I-S^*(z)S(z)\ge 0, \qquad z\in\bd, $$
$I$ is a unity matrix. A function  $s\in\css$ (a matrix function $S\in\css^{(m)}$) is said to be {\it inner (matrix) function} if its boundary values which exist
almost everywhere on the unit circle $\bt$, are unimodular (unitary). Given $s\in\css$, the Darlington synthesis problem asks whether there exists an
inner matrix function $S\in\css^{(m)}$ so that
\begin{equation}
S(z)=\|s_{ij}(z)\|_{i,j=1}^2: \quad s_{22}(z)=s(z).
\end{equation}

A seminal result of Arov \cite{ar71} and Douglas--Helton \cite{duhel} states that a Schur function admits the Darlington synthesis if and only if
it possesses a pseudocontinuation of bounded type across $\bt$. Recall that a meromorphic on a region $\Omega$ function of bounded type
is the quotient of two bounded (or contractive) analytic on $\Omega$ functions
\begin{equation}\label{nev}
f(z)=\frac{f_1(z)}{f_2(z)}\,, \qquad f_j\in\css(\Omega).
\end{equation}
Such functions constitute the Nevanlinna class $N(\Omega)$.

The goal of this note is to suggest a local version of the Darlington synthesis problem and to prove a local analog of the Arov--Douglas--Helton theorem.

\begin{definition}
Let $\g$ be an arc of the unit circle (the case $\g=\bt$ is not excluded). Denote by $\bd_e$ the exterior of the unit disk with respect to the
extended complex plane $\bar\bc$. A function $f\in N(\bd)$ admits the {\it pseudocontinuation of bounded type across $\g$} if there is a function
$\tilde f\in N(\bd_e)$ so that their boundary values agree
\begin{equation}
f(t)=\tilde f(t) \ \ {\rm a.e. \ on} \ \g.
\end{equation}
We write $f\in PC_\g$ for such functions. The class $PC_\g$ is nontrivial, see Example \ref{ex} in Section \ref{s1}.
\end{definition}

\begin{theorem}\label{th1}
Let $s\in\css$. The following conditions are equivalent.
\begin{enumerate}
  \item There is a matrix function $S=\|s_{ij}\|_{i,j=1}^2$ so that $s_{ij}\in\css$, $s_{22}=s$, and $S$ is unitary a.e. on the arc $\g$;
  \item $s\in PC_\g$.
\end{enumerate}
\end{theorem}

In the case $\g=\bt$, the above matrix function $S$ is inner due to the Maximum Norm Principle, and we come to
the Arov--Douglas--Helton theorem.

Given an arc $\g$, we denote by $\css_\g$ ($N_\g$) the class of the Schur (Nevanlinna) functions, unimodular a.e. on $\g$. Similarly, $\css^{(m)}_\g$
stands for the class of the Schur matrix functions unitary a.e. on $\g$.

It is clear that a matrix function with contractive entries does not necessarily belong to $\css^{(m)}$. So the question arises naturally whether
the matrix $S$ in Theorem \ref{th1} can be taken from $\css^{(m)}_\g$. If $s\in\css_\g$, the answer is affirmative: the matrix function
\begin{equation}\label{emb1}
S(z)=\begin{bmatrix}
s_{11}(z) & 0 \\
0 & s(z)
\end{bmatrix}
\end{equation}
with an arbitrary inner function $s_{11}$ belongs to $\css^{(m)}_\g$. But, in general, the answer is negative. The reason is that $s$ being an entry of a
contractive, nondiagonal matrix function is supposed to obey a {\it global} condition
\begin{equation}\label{logint}
\int_\bt \log(1-|s(t)|^2)\,m(dt)>-\infty,
\end{equation}
$m$ is the normalized Lebesgue measure on $\bt$. As it turns out, this condition is also sufficient.

\begin{theorem}\label{th2}
Let $s\in\css'_\g=\css\backslash\css_\g$. The following conditions are equivalent.
\begin{enumerate}
  \item There is a matrix function $V=\|v_{ij}\|_{i,j=1}^2\in\css^{(m)}_\g$ so that $v_{22}=s$;
  \item $s\in PC_\g$ and $\eqref{logint}$ holds.
\end{enumerate}
\end{theorem}

In contrast to the case $\g=\bt$ of the whole unit circle, we have neither the model spaces theory nor the Douglas--Shapiro--Shilds theorem at hand.
So the argument is more or less straightforward and relies upon the explicit (in a sense) expressions for the matrix entries of the matrices in question.

\section{Local pseudocontinuation and Darlington synthesis}
\label{s1}

Let us begin with the classes $\css_\g$ and $N_\g$, which play the same role as the class of inner functions does in the classical setting
of the Darlington synthesis problem.

\begin{example}\label{ex}
Let $a\in N_\g$. Write
\begin{equation}
\tilde a(\z):=\frac1{\ovl{a(1/\bar\z)}}\,, \quad \z\in\bd_e.
\end{equation}
Then $\tilde a\in N(\bd_e)$ and $\tilde a=a$ a.e. on $\g$, so $N_\g\subset PC_\g$. In particular, $\css_\g\subset PC_\g$, and, moreover,
such $s\in PC_\g\backslash PC_\bt$ unless $s$ is an inner function.
\end{example}

\smallskip
\noindent
{\it Proof of Theorem \ref{th1}}.

$(1)\Rightarrow (2)$. The argument here is standard. By the hypothesis, $\det S\noe0$, so we write
\begin{equation}
\begin{split}
U(\z) &:=(S^{-1})^*(1/\bar\z) \\ &=\frac1{\ovl{\det S(1/\bar\z)}}\,
\begin{bmatrix}
\ovl{s(1/\bar\z)} & -\ovl{s_{21}(1/\bar\z)} \\
-\ovl{s_{12}(1/\bar\z)} & \ovl{s_{11}(1/\bar\z)}
\end{bmatrix}, \quad \z\in\bd_e.
\end{split}
\end{equation}
It is clear that all entries of $U$ belong to $N(\bd_e)$, and $U=S$ a.e. on $\g$. Hence, $s$ admits the pseudocontinuation of bounded type across $\g$,
$s\in PC_\g$, with
$$ \tilde s(\z)=\frac{\ovl{s_{11}(1/\bar\z)}}{\ovl{\det S(1/\bar\z)}}\,. $$
Note that in fact each entry of the bounded matrix function $S$, unitary a.e. on $\g$, is in the class $PC_\g$.

$(2)\Rightarrow (1)$. The arguments in Example \ref{ex} and around relation \eqref{emb1} show that the result holds for $s\in\css_\g$.
So we assume further that $s\in\css'_\g=\css\backslash\css_\g$.

Define a pair of functions on $\bd$
\begin{equation}\label{pair}
g(z):=\ovl{\tilde s (1/\bar z)}\in N(\bd), \qquad h(z):=1-g(z)s(z)\in N(\bd),
\end{equation}
where $\tilde s$ is the pseudocontinuation of bounded type of $s$ across $\g$. Now, $s\notin\css_\g$ implies $h\noe0$, so
$\log|h|\in L^1(\bt)$, see \cite[Theorem 2.2]{Du70}, and
\begin{equation}\label{loclog}
\int_\g \log|h(t)|\,m(dt)=\int_\g \log(1-|s(t)|^2)\,m(dt)>-\infty.
\end{equation}
We see that $\log(1-|s|^2)\in L^1(\g)$ as long as $s\in PC_\g\backslash\css_\g$, which is a local counterpart of relation \eqref{logint}.

In view of \eqref{loclog}, the function
\begin{equation}\label{locout}
\s_\g(z):=\exp\Bigl\{\frac12\,\int_{\g} \frac{t+z}{t-z}\,\log(1-|s(t)|^2)\,m(dt)\Bigr\}
\end{equation}
is a well-defined, outer Schur function, $\s_\g\in\css_{out}$, with the boundary values
\begin{equation}\label{bouval1}
|\s_\g(t)|^2=1-|s(t)|^2 \ \ {\rm a.e. \  on} \ \g, \quad |\s_\g(t)|=1\ \ {\rm a.e. \  on} \ \g':=\bt\backslash\g.
\end{equation}
We choose $s_{12}=\s_\g$.

Going back to the Nevanlinna functions $g,h$ in \eqref{pair}, we write
\begin{equation*}
\begin{split}
g(z) &=\frac{g_1(z)}{g_2(z)}=\frac{I_{g_1}(z)O_{g_1}(z)}{I_{g_2}(z)O_{g_2}(z)}\,, \quad g_j\in\css, \\
h(z) &=\frac{h_1(z)}{h_2(z)}=\frac{I_{h_1}(z)O_{h_1}(z)}{I_{h_2}(z)O_{h_2}(z)}\,, \quad h_j\in\css,
\end{split}
\end{equation*}
where $f=I_fO_f$ is the standard inner-outer factorization of a Schur function $f$. We proceed with the further factorization of the outer factors
with respect to $\g$, precisely,
\begin{equation}\label{fact}
\begin{split}
O(z) &=O(z,\g)O(z,\g'), \\
O(z,\gg) &:=\exp\Bigl\{\frac12\,\int_{\gg} \frac{t+z}{t-z}\,\log|O(t)|\,m(dt)\Bigr\}
\end{split}
\end{equation}
for the arc $\gg=\g,\g'$. We have $O(\cdot,\gg)\in\css_{out}$ and
\begin{equation}\label{bouval2}
|O(t,\g)|=1 \ \ {\rm a.e. \  on} \ \g', \quad |O(t,\g')|=1 \ \ {\rm a.e. \  on} \ \g.
\end{equation}
Hence,
\begin{equation}\label{fact1}
\begin{split}
g(z) &=\frac{I_{g_1}(z)O_{g_1}(z,\g)O_{g_1}(z,\g')}{I_{g_2}(z)O_{g_2}(z,\g)O_{g_2}(z,\g')}\,,  \\
h(z) &=\frac{I_{h_1}(z)O_{h_1}(z,\g)O_{h_1}(z,\g')}{I_{h_2}(z)O_{h_2}(z,\g)O_{h_2}(z,\g')}\,.
\end{split}
\end{equation}

Put
\begin{equation}
p(z):=I_{g_2}(z)I_{h_2}(z)\,O_{g_2}(z,\g')O_{h_2}(z,\g'),
\end{equation}
so $|p|=1$ a.e. on $\g$. Our choice of $s_{11}$ and $s_{21}$ is
\begin{equation}\label{1stcol}
\begin{split}
s_{11}(z) &= -p(z)g(z)=-I_{g_1}(z)I_{h_2}(z)\,O_{g_1}(z,\g')O_{h_2}(z,\g')\,\frac{O_{g_1}(z,\g)}{O_{g_2}(z,\g)}\,, \\
s_{21}(z) &=p(z)\,\frac{h(z)}{\s_\g(z)}=I_{g_2}(z)I_{h_1}(z)\,O_{g_2}(z,\g')O_{h_1}(z,\g')\,\frac{O_{h_1}(z,\g)}{O_{h_2}(z,\g)\s_\g(z)}\,.
\end{split}
\end{equation}

It is clear that $s_{22}=s$ and $s_{12}=\s_\g$ \eqref{locout} are contractive functions. As for $s_{11}$ and $s_{21}$ \eqref{1stcol}, we note that
they belong to an important subclass $N^+(\bd)\subset N(\bd)$ of the Nevanlinna class, which is usually referred to as the Smirnov class, see
\cite[Section 2.5]{Du70}. It is characterized by the denominator in \eqref{nev} being an outer Schur function, which is exactly the case in \eqref{1stcol}.
The main feature of this class is the Smirnov maximum modulus principle, \cite[Theorem 2.11]{Du70},
\begin{equation}\label{maxmod}
f\in N^+(\bd), \quad |f(t)|\le1  \ \ {\rm a.e. \  on} \ \bt \ \Rightarrow \ f\in\css.
\end{equation}
For $t\in\g'$ we have $|s_{11}|\le1$, $|s_{21}|\le1$ a.e. in view of \eqref{bouval2}. For $t\in\g$ we have $|p|=1$ a.e., so
\begin{equation*}
|s_{11}(t)|\le |g(t)|=|s(t)|\le1, \quad |s_{21}(t)|\le\frac{|h(t)|}{|\s_\g(t)|}=(1-|s(t)|^2)^{1/2}\le1,
\end{equation*}
and the first claim of the Theorem follows from \eqref{maxmod}.

To show that $S$ is unitary a.e. on $\g$, we put
\begin{equation*}
S^*(t)S(t)=\begin{bmatrix}
|s_{11}(t)|^2+|s_{21}(t)|^2 &  \ovl{s_{11}(t)}s_{12}(t)+\ovl{s_{21}(t)}s(t) \\
s_{11}(t)\ovl{s_{12}(t)}+s_{21}(t)\ovl{s(t)} & |s_{12}(t)|^2+|s(t)|^2
\end{bmatrix}.
\end{equation*}
By \eqref{bouval1},
$$ |s_{12}(t)|^2+|s(t)|^2=|\s_\g(t)|^2+|s(t)|^2=1-|s(t)|^2+|s(t)|^2=1 $$
a.e. on $\g$. Next, $|p|=1$ a.e. on $\g$ implies
$$ |s_{11}(t)|^2+|s_{21}(t)|^2=|g(t)|^2+\frac{|h(t)|^2|}{|\s_\g(t)|^2}=|s(t)|^2 +1-|s(t)|^2=1. $$
Finally, by \eqref{bouval1} and the definition of $h$,
\begin{equation*}
\begin{split}
s_{11}(t)\ovl{s_{12}(t)}+s_{21}(t)\ovl{s(t)} &=p(t)\Bigl(-g(t)\ovl{\s_\g(t)}+\frac{h(t)}{\s_\g(t)}\,\ovl{s(t)}\Bigr) \\
&=  p(t)\ovl{s(t)}\Bigl(-\ovl{\s_\g(t)}+\frac{h(t)}{\s_\g(t)}\Bigr)=0
\end{split}
\end{equation*}
a.e. on $\g$. So, $S^*S=I$, as claimed. The proof is complete.  \hfill  $\Box$

\smallskip
\noindent
{\it Proof of Theorem \ref{th2}}.

$(1)\Rightarrow (2)$. By Theorem \ref{th1}, $s\in PC_\g$, so we have to verify condition \eqref{logint}. Note that at least one of the functions
$v_{12}$, $v_{21}$ is not identically zero (otherwise, $s\in\css_\g$). Assume that  $v_{12}\noe0$ and write
$$ I-V^*(t)V(t)=\begin{bmatrix}
* & * \\
* & 1-|v_{12}(t)|^2-|s(t)|^2
\end{bmatrix}\ge0, \qquad t\in\bt, $$
so $1-|s(t)|^2\ge |v_{12}(t)|^2$. Since $\log |v_{12}|\in L^1(\bt)$, condition \eqref{logint} follows.

$(2)\Rightarrow (1)$. The matrix $V$ arises as an appropriate modification of the matrix $S$ from Theorem \ref{th1}.
By \eqref{logint}, the function
\begin{equation*}
\s_{\g'}(z):=\exp\Bigl\{\frac12\,\int_{\g'} \frac{t+z}{t-z}\,\log(1-|s(t)|^2)\,m(dt)\Bigr\}
\end{equation*}
is well-defined and lies in $\css_{out}$. Denote by $e$ the outer Schur function with
\begin{equation*}
|e(t)|=1 \ \ {\rm a.e. \  on} \ \g, \qquad |e(t)|=\eps \ \ {\rm a.e. \  on} \ \g',
\end{equation*}
where $0<\eps<1/3$ is a small enough positive constant, and put $r:=e\s_{\g'}$.
Take the matrix $V$ in question as
\begin{equation*}
\begin{split}
V(z) &=\begin{bmatrix}
r(z) & 0 \\
0    & 1
\end{bmatrix}  S(z)
\begin{bmatrix}
r(z) & 0 \\
0    & 1
\end{bmatrix}=
\begin{bmatrix}
r^2(z)s_{11}(z) & r(z)s_{12}(z) \\
r(z)s_{21}(z)    & s(z)
\end{bmatrix},
\end{split}
\end{equation*}
As both $e$ and $\s_{\g'}$ are unimodular on $\g$, then so is $r$, and thereby $V$ is unitary a.e. on $\g$.

It remains to check that $V\in\css^{(m)}$. To this end we put on the arc $\g'$
\begin{equation*}
W=
\begin{bmatrix}
w_{11} & w_{12} \\
w_{21} & w_{22}
\end{bmatrix}:=I -V^*V =
\begin{bmatrix}
1-|r^2s_{11}|^2-|rs_{21}|^2 & -\bar r|r|^2s_{12}\ovl{s_{11}}-\ovl{rs_{21}}s \\
-r|r|^2 s_{11}\ovl{s_{12}}-rs_{21}\ovl{s} & 1-|rs_{12}|^2-|s|^2
\end{bmatrix}.
\end{equation*}
Since
$$  r(z)s_{12}(z)=e(z)\s_{\g'}(z)\s_\g(z)=e(z)\,\exp\Bigl\{\frac12\,\int_{\bt} \frac{t+z}{t-z}\,\log(1-|s(t)|^2)\,m(dt)\Bigr\}\,, $$
then $|rs_{12}|^2=\eps^2(1-|s|^2)$ and so
\begin{equation}
w_{22}(t)=1-|r(t)s_{12}(t)|^2-|s(t)|^2=(1-\eps^2)(1-|s(t)|^2)\ge0
\end{equation}
a.e. on $\g'$. Next, the functions $r, s_{11}, s_{21}$ are contractive, so
\begin{equation*}
\begin{split}
w_{11}(t) &=1-|r^2(t)s_{11}(t)|^2-|r(t)s_{21}(t)|^2 =1-|r(t)|^2\bigl(|r(t)s_{11}(t)|^2+|s_{21}(t)|^2\bigr) \\
&=1-\eps^2(1-|s(t)|^2)(|r(t)s_{11}(t)|^2+|s_{21}(t)|^2)\ge 1-2\eps^2(1-|s(t)|^2)
\end{split}
\end{equation*}
and
\begin{equation}
w_{11}(t)\ge 1-2\eps^2(1-|s(t)|^2)>\frac79
\end{equation}
a.e. on $\g'$ for $0<\eps<1/3$. Finally,
$$ -w_{21}(t)=r(t)\bigl(s_{21}(t)\ovl{s(t)}+ |r(t)|^2 s_{11}(t)\ovl{s_{12}(t)}\bigr)=r(t)v(t),\quad |v(t)|\le2, $$
and so
\begin{equation*}
W(t)\ge
\begin{bmatrix}
\frac79 & -\ovl{r(t)v(t)} \\
-r(t)v(t) & (1-\eps^2)(1-|s(t)|^2)
\end{bmatrix}=\widetilde W(t)=\|\tilde w_{ij}(t)\|_{i,j=1}^2
\end{equation*}
a.e. on $\g'$.

To show that $\widetilde W\ge0$ a.e. on $\g'$, given $\tilde w_{11}\ge0$, $\tilde w_{22}\ge0$, we compute the determinant of $\widetilde W$
\begin{equation*}
\begin{split}
\tilde w_{11}(t)\tilde w_{22}(t)-|\tilde w_{12}(t)|^2 &=\frac79\,(1-\eps^2)(1-|s(t)|^2)-|\eps v(t)|^2(1-|s(t)|^2) \\
&\ge\Bigl(\frac79\,(1-\eps^2)-4\eps^2\Bigr)(1-|s(t)|^2)\ge \frac29\,(1-|s(t)|^2)\ge0
\end{split}
\end{equation*}
a.e. on $\g'$ for $0<\eps<1/3$. So, $V\in\css^{(m)}_\g$, as claimed. \hfill $\Box$

\smallskip

We complete this note with some properties of the pseudocontinuation of bounded type across an arc.

\begin{proposition}\label{pr1}
Let $s_1,s_2\in\css$ and $|s_1|=|s_2|$ a.e. on the arc $\g$. Then $s_1$ and $s_2$ belong to $PC_\g$ simultaneously.
\end{proposition}
\begin{proof}
Let $s_1\in PC_\g$. We have the canonical factorization
$$ s_k(z)=I_k(z)O_k(z,\g)O_k(z,\g'), \qquad k=1,2, $$
and, by the assumption, $O_1(\cdot,\g)=O_2(\cdot,\g)$. Hence,
$$ s_2(z)=a(z)s_1(z), \qquad a(z):=\frac{I_2(z)O_2(z,\g')}{I_1(z)O_1(z,\g')}\,. $$
The function $a\in N_\g$, so, see Example \ref{ex}, $a\in PC_\g$. The later class is closed under multiplication, so $s_2\in PC_\g$,
as claimed.
\end{proof}

\smallskip

Recall that $\s_\g$ is defined in \eqref{locout} under condition \eqref{loclog}.

\begin{proposition}\label{pr2}
Let $s\in\css$ and $\log(1-|s|^2)\in L^1(\g)$. Then
\begin{equation*}
s\in PC_\g \ \Leftrightarrow \ \s_\g\in PC_\g.
\end{equation*}
\end{proposition}
\begin{proof}
As we mentioned earlier in the proof of Theorem \ref{th1}, each entry of the bounded matrix function $S$, unitary a.e. on $\g$,
is in the class $PC_\g$. If $s\in PC_\g$, the matrix function $S$ in Theorem \ref{th1} contains both $s$ and $\s_\g$ as its entries,
and we are done.

Conversely, let $\s_\g\in PC_\g$. By Theorem \ref{th1}, there is a matrix function $\ss$ with contractive entries, unitary a.e. on $\g$, and
$$ \ss(z)=\begin{bmatrix}
\s_{11}(z) & \s_{12}(z) \\
\s_{21}(z) & \s_\g(z)
\end{bmatrix}. $$
In particular, $|\s_{12}|^2+|\s_\g|^2=1$, and so $|\s_{12}|=|s|$ a.e. on $\g$. The function $\s_{12}$, being the entry of $\ss$,
belongs to the class $PC_\g$. By Proposition \ref{pr1}, so does $s$, as claimed.
\end{proof}

\begin{remark}
The fact that $\g$ is the arc of the unit circle is obviously immaterial. The argument works for an arbitrary Borel set $\g\subset\bt$ of positive
measure.
\end{remark}

\smallskip

{\bf Acknowledgement}. The author thanks the participants of the Analysis Seminar at Kharkov National University for valuable discussions.

\end{document}